\documentclass[12pt]{amsart}      
\usepackage{amssymb}
\usepackage{eucal}
\usepackage{amsmath}
\usepackage{amscd}
\usepackage[dvips]{color}
\usepackage{multicol}
\usepackage[all]{xy}           
\usepackage{graphicx}
\usepackage{color}
\usepackage{colordvi}
\usepackage{xspace}

\begin{document}  

\newcommand{\nc}{\newcommand}
\newcommand{\delete}[1]{}
\nc{\dfootnote}[1]{{}}          
\nc{\ffootnote}[1]{\dfootnote{#1}}
\nc{\mfootnote}[1]{\footnote{#1}} 
\nc{\ofootnote}[1]{\footnote{\tiny Older version: #1}} 

\nc{\mlabel}[1]{\label{#1}}  
\nc{\mcite}[1]{\cite{#1}}  
\nc{\mref}[1]{\ref{#1}}  

\delete{
\nc{\mlabel}[1]{\label{#1}  
{\hfill \hspace{1cm}{\bf{{\ }\hfill(#1)}}}}
\nc{\mcite}[1]{\cite{#1}{{\bf{{\ }(#1)}}}}  
\nc{\mref}[1]{\ref{#1}{{\bf{{\ }(#1)}}}}  
}

\nc{\mbibitem}[1]{\bibitem{#1}} 

\newtheorem{theorem}{Theorem}[section]
\newtheorem{prop}[theorem]{Proposition}
\newtheorem{defn}[theorem]{Definition}
\newtheorem{lemma}[theorem]{Lemma}
\newtheorem{coro}[theorem]{Corollary}
\newtheorem{prop-def}{Proposition-Definition}[section]
\newtheorem{claim}{Claim}[section]
\newtheorem{remark}[theorem]{Remark}
\newtheorem{propprop}{Proposed Proposition}[section]
\newtheorem{conjecture}{Conjecture}
\newtheorem{exam}{Example}[section]
\newtheorem{assumption}{Assumption}
\newtheorem{condition}[theorem]{Assumption}
\newtheorem{question}[theorem]{Question}

\renewcommand{\labelenumi}{{\rm(\alph{enumi})}}
\renewcommand{\theenumi}{\alph{enumi}}

\nc{\tred}[1]{\textcolor{red}{#1}}
\nc{\tblue}[1]{\textcolor{blue}{#1}}
\nc{\tgreen}[1]{\textcolor{green}{#1}}
\nc{\tpurple}[1]{\textcolor{purple}{#1}}
\nc{\btred}[1]{\textcolor{red}{\bf #1}}
\nc{\btblue}[1]{\textcolor{blue}{\bf #1}}
\nc{\btgreen}[1]{\textcolor{green}{\bf #1}}
\nc{\btpurple}[1]{\textcolor{purple}{\bf #1}}

\nc{\adec}{\check{;}}
\nc{\dftimes}{\widetilde{\otimes}} \nc{\dfl}{\succ}
\nc{\dfr}{\prec} \nc{\dfc}{\circ} \nc{\dfb}{\bullet}
\nc{\dft}{\star} \nc{\dfcf}{{\mathbf k}} \nc{\spr}{\cdot}
\nc{\disp}[1]{\displaystyle{#1}}
\nc{\bin}[2]{ (_{\stackrel{\scs{#1}}{\scs{#2}}})}  
\nc{\binc}[2]{ \left (\!\! \begin{array}{c} \scs{#1}\\
    \scs{#2} \end{array}\!\! \right )}  
\nc{\bincc}[2]{  \left ( {\scs{#1} \atop
    \vspace{-.5cm}\scs{#2}} \right )}  
\nc{\sarray}[2]{\begin{array}{c}#1 \vspace{.1cm}\\ \hline
    \vspace{-.35cm} \\ #2 \end{array}}
\nc{\bs}{\bar{S}} \nc{\dcup}{\stackrel{\bullet}{\cup}}
\nc{\dbigcup}{\stackrel{\bullet}{\bigcup}} \nc{\etree}{\big |}
\nc{\la}{\longrightarrow} \nc{\fe}{\'{e}} \nc{\rar}{\rightarrow}
\nc{\dar}{\downarrow} \nc{\dap}[1]{\downarrow
\rlap{$\scriptstyle{#1}$}} \nc{\uap}[1]{\uparrow
\rlap{$\scriptstyle{#1}$}} \nc{\defeq}{\stackrel{\rm def}{=}}
\nc{\dis}[1]{\displaystyle{#1}} \nc{\dotcup}{\,
\displaystyle{\bigcup^\bullet}\ } \nc{\sdotcup}{\tiny{
\displaystyle{\bigcup^\bullet}\ }} \nc{\hcm}{\ \hat{,}\ }
\nc{\hcirc}{\hat{\circ}} \nc{\hts}{\hat{\shpr}}
\nc{\lts}{\stackrel{\leftarrow}{\shpr}}
\nc{\rts}{\stackrel{\rightarrow}{\shpr}} \nc{\lleft}{[}
\nc{\lright}{]} \nc{\uni}[1]{\tilde{#1}} \nc{\wor}[1]{\check{#1}}
\nc{\free}[1]{\bar{#1}} \nc{\den}[1]{\check{#1}} \nc{\lrpa}{\wr}
\nc{\curlyl}{\left \{ \begin{array}{c} {} \\ {} \end{array}
    \right .  \!\!\!\!\!\!\!}
\nc{\curlyr}{ \!\!\!\!\!\!\!
    \left . \begin{array}{c} {} \\ {} \end{array}
    \right \} }
\nc{\leaf}{\ell}       
\nc{\longmid}{\left | \begin{array}{c} {} \\ {} \end{array}
    \right . \!\!\!\!\!\!\!}
\nc{\ot}{\otimes} \nc{\sot}{{\scriptstyle{\ot}}}
\nc{\otm}{\overline{\ot}}
\nc{\ora}[1]{\stackrel{#1}{\rar}}
\nc{\ola}[1]{\stackrel{#1}{\la}}
\nc{\scs}[1]{\scriptstyle{#1}} \nc{\mrm}[1]{{\rm #1}}
\nc{\margin}[1]{\marginpar{\rm #1}}   
\nc{\dirlim}{\displaystyle{\lim_{\longrightarrow}}\,}
\nc{\invlim}{\displaystyle{\lim_{\longleftarrow}}\,}
\nc{\mvp}{\vspace{0.5cm}} \nc{\svp}{\vspace{2cm}}
\nc{\vp}{\vspace{8cm}} \nc{\proofbegin}{\noindent{\bf Proof: }}
\nc{\proofend}{$\blacksquare$ \vspace{0.5cm}}
\nc{\sha}{{\mbox{\cyr X}}}  
\nc{\ncsha}{{\mbox{\cyr X}^{\mathrm NC}}} \nc{\ncshao}{{\mbox{\cyr
X}^{\mathrm NC,\,0}}}
\nc{\shpr}{\diamond}    
\nc{\shprm}{\overline{\diamond}}    
\nc{\shpro}{\diamond^0}    
\nc{\shprr}{\diamond^r}     
\nc{\shpra}{\overline{\diamond}^r}
\nc{\shpru}{\check{\diamond}} \nc{\catpr}{\diamond_l}
\nc{\rcatpr}{\diamond_r} \nc{\lapr}{\diamond_a}
\nc{\sqcupm}{\ot}
\nc{\lepr}{\diamond_e} \nc{\vep}{\varepsilon} \nc{\labs}{\mid\!}
\nc{\rabs}{\!\mid} \nc{\hsha}{\widehat{\sha}}
\nc{\lsha}{\stackrel{\leftarrow}{\sha}}
\nc{\rsha}{\stackrel{\rightarrow}{\sha}} \nc{\lc}{\lfloor}
\nc{\rc}{\rfloor} \nc{\sqmon}[1]{\langle #1\rangle}
\nc{\forest}{\calf} \nc{\ass}[1]{\alpha({#1})}
\nc{\altx}{\Lambda_X} \nc{\vecT}{\vec{T}} \nc{\onetree}{\bullet}
\nc{\Ao}{\check{A}}
\nc{\seta}{\underline{\Ao}}
\nc{\deltaa}{\overline{\delta}}
\nc{\trho}{\tilde{\rho}}

\nc{\mmbox}[1]{\mbox{\ #1\ }}
\nc{\ann}{\mrm{ann}}
\nc{\Aut}{\mrm{Aut}} \nc{\can}{\mrm{can}} \nc{\colim}{\mrm{colim}}
\nc{\Cont}{\mrm{Cont}} \nc{\rchar}{\mrm{char}}
\nc{\cok}{\mrm{coker}} \nc{\dtf}{{R-{\rm tf}}} \nc{\dtor}{{R-{\rm
tor}}}
\renewcommand{\det}{\mrm{det}}
\nc{\depth}{{\mrm d}}
\nc{\Div}{{\mrm Div}} \nc{\End}{\mrm{End}} \nc{\Ext}{\mrm{Ext}}
\nc{\Fil}{\mrm{Fil}} \nc{\Frob}{\mrm{Frob}} \nc{\Gal}{\mrm{Gal}}
\nc{\GL}{\mrm{GL}} \nc{\Hom}{\mrm{Hom}} \nc{\hsr}{\mrm{H}}
\nc{\hpol}{\mrm{HP}} \nc{\id}{\mrm{id}} \nc{\im}{\mrm{im}}
\nc{\incl}{\mrm{incl}} \nc{\length}{\mrm{length}}
\nc{\LR}{\mrm{LR}} \nc{\mchar}{\rm char} \nc{\NC}{\mrm{NC}}
\nc{\mpart}{\mrm{part}} \nc{\pl}{\mrm{PL}}
\nc{\ql}{{\QQ_\ell}} \nc{\qp}{{\QQ_p}}
\nc{\rank}{\mrm{rank}} \nc{\rba}{\rm{RBA }} \nc{\rbas}{\rm{RBAs }}
\nc{\rbpl}{\mrm{RBPL}}
\nc{\rbw}{\rm{RBW }} \nc{\rbws}{\rm{RBWs }} \nc{\rcot}{\mrm{cot}}
\nc{\rest}{\rm{controlled}\xspace}
\nc{\rdef}{\mrm{def}} \nc{\rdiv}{{\rm div}} \nc{\rtf}{{\rm tf}}
\nc{\rtor}{{\rm tor}} \nc{\res}{\mrm{res}} \nc{\SL}{\mrm{SL}}
\nc{\Spec}{\mrm{Spec}} \nc{\tor}{\mrm{tor}}
\nc{\supp}{\mrm{supp}}
\nc{\Tr}{\mrm{Tr}}
\nc{\mtr}{\mrm{sk}}

\nc{\ab}{\mathbf{Ab}} \nc{\Alg}{\mathbf{Alg}}
\nc{\Algo}{\mathbf{Alg}^0} \nc{\Bax}{\mathbf{Bax}}
\nc{\Baxo}{\mathbf{Bax}^0} \nc{\RB}{\mathbf{RB}}
\nc{\RBo}{\mathbf{RB}^0} \nc{\BRB}{\mathbf{RB}}
\nc{\Dend}{\mathbf{DD}} \nc{\bfk}{{\bf k}} \nc{\bfone}{{\bf 1}}
\nc{\base}[1]{{a_{#1}}} \nc{\detail}{\marginpar{\bf More detail}
    \noindent{\bf Need more detail!}
    \svp}
\nc{\Diff}{\mathbf{Diff}} \nc{\gap}{\marginpar{\bf
Incomplete}\noindent{\bf Incomplete!!}
    \svp}
\nc{\FMod}{\mathbf{FMod}} \nc{\mset}{\mathbf{MSet}}
\nc{\rb}{\mathrm{RB}} \nc{\Int}{\mathbf{Int}}
\nc{\Mon}{\mathbf{Mon}}
\nc{\remarks}{\noindent{\bf Remarks: }} \nc{\Rep}{\mathbf{Rep}}
\nc{\Rings}{\mathbf{Rings}} \nc{\Sets}{\mathbf{Sets}}
\nc{\DT}{\mathbf{DT}}

\nc{\BA}{{\mathbb A}} \nc{\CC}{{\mathbb C}} \nc{\DD}{{\mathbb D}}
\nc{\EE}{{\mathbb E}} \nc{\FF}{{\mathbb F}} \nc{\GG}{{\mathbb G}}
\nc{\HH}{{\mathbb H}} \nc{\LL}{{\mathbb L}} \nc{\NN}{{\mathbb N}}
\nc{\QQ}{{\mathbb Q}} \nc{\RR}{{\mathbb R}} \nc{\TT}{{\mathbb T}}
\nc{\VV}{{\mathbb V}} \nc{\ZZ}{{\mathbb Z}}


\nc{\cala}{{\mathcal A}} \nc{\calc}{{\mathcal C}}
\nc{\cald}{{\mathcal D}} \nc{\cale}{{\mathcal E}}
\nc{\calf}{{\mathcal F}} \nc{\calfr}{{{\mathcal F}^{\,r}}}
\nc{\calfo}{{\mathcal F}^0} \nc{\calfro}{{\mathcal F}^{\,r,0}}
\nc{\oF}{\overline{F}}  \nc{\calg}{{\mathcal G}}
\nc{\calh}{{\mathcal H}} \nc{\cali}{{\mathcal I}}
\nc{\calj}{{\mathcal J}} \nc{\call}{{\mathcal L}}
\nc{\calm}{{\mathcal M}} \nc{\caln}{{\mathcal N}}
\nc{\calo}{{\mathcal O}} \nc{\calp}{{\mathcal P}}
\nc{\calr}{{\mathcal R}} \nc{\calt}{{\mathcal T}}
\nc{\caltr}{{\mathcal T}^{\,r}}
\nc{\calu}{{\mathcal U}} \nc{\calv}{{\mathcal V}}
\nc{\calw}{{\mathcal W}} \nc{\calx}{{\mathcal X}}
\nc{\CA}{\mathcal{A}}

\nc{\fraka}{{\mathfrak a}} \nc{\frakB}{{\mathfrak B}}
\nc{\frakb}{{\mathfrak b}} \nc{\frakd}{{\mathfrak d}}
\nc{\oD}{\overline{D}}
\nc{\frakF}{{\mathfrak F}} \nc{\frakg}{{\mathfrak g}}
\nc{\frakm}{{\mathfrak m}} \nc{\frakM}{{\mathfrak M}}
\nc{\frakMo}{{\mathfrak M}^0} \nc{\frakp}{{\mathfrak p}}
\nc{\frakS}{{\mathfrak S}} \nc{\frakSo}{{\mathfrak S}^0}
\nc{\fraks}{{\mathfrak s}} \nc{\os}{\overline{\fraks}}
\nc{\frakT}{{\mathfrak T}}
\nc{\oT}{\overline{T}}
\nc{\frakX}{{\mathfrak X}} \nc{\frakXo}{{\mathfrak X}^0}
\nc{\frakx}{{\mathbf x}}
\nc{\frakTx}{\frakT}      
\nc{\frakTa}{\frakT^a}        
\nc{\frakTxo}{\frakTx^0}   
\nc{\caltao}{\calt^{a,0}}   
\nc{\ox}{\overline{\frakx}} \nc{\fraky}{{\mathfrak y}}
\nc{\frakz}{{\mathfrak z}} \nc{\oX}{\overline{X}}

\font\cyr=wncyr10

\nc{\redtext}[1]{\textcolor{red}{#1}}


\title{Rota-Baxter operators on generalized power series rings}
\author{Li Guo}
\address{Department of Mathematics and Computer Science,
         Rutgers University,
         Newark, NJ 07102}
\email{liguo@rutgers.edu}
\author{Zhongkui Liu}
\address{Department of Mathematics,
    Northwest Normal University,
    Lanzhou, Gansu, China}
\email{liuzk@nwnu.edu.cn}



\begin{abstract}
An important instance of Rota-Baxter algebras from their quantum field theory application is the ring of Laurent series with a suitable projection. We view the ring of Laurent series as a special case of generalized power series rings with exponents in an ordered monoid. We study when a generalized power series ring has a Rota-Baxter operator and how this is related to the ordered monoid.
\\
\\
{\em Key Words:} Rota-Baxter algebra, generalized power series ring, strictly ordered monoid.

\end{abstract}


\maketitle


\setcounter{section}{0}
{\ }
\vspace{-1cm}

\section{Introduction}
Let $R$ be a unitary commutative ring. A Rota-Baxter $R$-algebra (of weight -1) is an associative $R$-algebra $A$ together with a linear operator $P:A\to A$ such that
\begin{equation}
P(x)P(y)=P(xP(y))+P(P(x)y)- P(xy), \forall\, x,y\in A.
\mlabel{eq:rba}
\end{equation}
Since its introduction by G. Baxter~\mcite{Ba} in 1960 from his probability study, it has been studied by many authors such as P. Cartier, A. Connes, D. Kreimer, J.-L. Loday and G.-C. Rota, in connection with combinatorics, mathematics physics, operads and number theory~\mcite{Ag3,A-M,A-G-K-O,Ca,C-K0,C-K1,E-G4,E-G6,Gu2, Gu5,G-K1,G-K2,G-S,G-Z,Le1,L-R1,Ro1,Ro2}.
A simple but important example of a Rota-Baxter algebra is the algebra of Laurent series
$$R [[\vep, \vep^{-1}]=\left \{ \sum_{n=k}^\infty a_n \vep^{n} \
\big |\ a_n\in R, -\infty <k<\infty \right \}$$
where the Rota-Baxter operator is the projection
$$ P(\sum_{n=k}^\infty a_n \vep^n)= \sum_{n< 0} a_n \vep^n.$$
See \mcite{C-K0,C-K1,C-M,E-G-K3,E-G-M} for its application in the renormalization of quantum field theory.

A natural algebraic setting to view the ring of Laurent series is to regard it as a special case of generalized power series rings with exponents in a strictly ordered monoid. They have been studied extensively since the 1990s~\mcite{E-R,Liu,L-A,Ri1,Ri2,Ri3,Ri4}.
In this context, the above operator $P$ is obtained as a cut-off operator (whose general definition will be given in Section~\mref{sec:cutoff}) at the point $0$ of the ordered monoid $(\ZZ,\leq)$: Let
$$ P(\sum_{n=k}^\infty a_n \vep^n)=\sum_{n=k}^\infty
    \check{a}_n \vep^n,$$
then
\begin{equation}
\check{a}_n= \begin{cases} a_n, & n<0,\\ 0, & n\not< 0
\end{cases}
\mlabel{eq:lacut}
\end{equation}
Motivated by this connection, we naturally ask when a cut-off operator in a generalized power series ring is a Rota-Baxter operator. As it turns out, the answer to this question is related to a more general study of the relationship between Rota-Baxter algebras and algebras of generalized power series. We will present our findings in Section~\mref{sec:gps}. In particular, we show that a decomposition of the monoid gives a Rota-Baxter operator on the algebra of generalized power series. This gives a quite large class of Rota-Baxter algebras. This also generalizes the classical construction of Rota-Baxter operator on a semigroup algebra from a decomposition of the semigroup. Then in Section~\mref{sec:cutoff}, we use these general results to deduce the answer to our question above.

\section{Decomposition in generalized power series rings}
\mlabel{sec:gps}

We recall the concept of a strictly ordered monoid. See \mcite{E-R,Liu,L-A,Ri1,Ri2,Ri3,Ri4} for further details.
By an ordered set, we will mean a partially ordered set. An ordered set $(S,\leq)$ is called {\bf artinian}
if every strictly decreasing sequence of elements of $S$ is finite, and is called {\bf narrow} if every subset of pairwise
order-incomparable elements of $S$ is finite. Let $S$ be a commutative monoid. Unless stated otherwise, the operation of $S$
shall be denoted additively, and the neutral element by $0$.

An ordered monoid $(S,\leq)$ is called a {\bf strictly ordered monoid} if for $s,s',t\in S$
and $s<s'$, we have $s+t<s'+t$. Let $R$ be a commutative  ring. Let $
[[R^{S,\leq}]]$ be the set of all maps $f:S\longrightarrow R$ such
that the ordered set $\supp(f)=\{s\in S\vert f(s)\neq 0\}$ is artinian and narrow.
With pointwise addition, $[[R^{S,\leq}]]$ is an abelian additive group.
For every $s\in S$ and $f,g\in [[R^{S,\leq}]]$, let $$X_s(f,g)=\{(u,v)\in S\times S\vert
s=u+v, f(u)\neq 0, g(v)\neq 0\}.$$
It follows from \cite[1.16]{Ri3} that
$X_s(f,g)$ is finite. This fact allows us to define the operation of
convolution:
$$(fg)(s)=\sum_{(u,v)\in X_s(f,g)}f(u)g(v).$$
With this operation, and pointwise addition, $[[R^{S,\leq}]]$
becomes a commutative ring,
 which is called the {\bf ring of generalized power series}. The elements of $[[R^{S,\leq}]]$
are called generalized power series with coefficients in $R$ and
exponents in $S$.

For example, if $S=\Bbb N\cup\{0\}$ and $\leq$ is  the usual order,
then $[[R^{\Bbb N\cup\{0\}, \leq}]]$ is isomorphic to $R[[x]]$, the usual ring of
power series. If $S$ is a commutative monoid and $\leq$ is the
trivial order, then $[[R^{S,\leq}]]=R[S]$, the monoid-ring of $S$
over $R$. If $S=\Bbb Z$ and $\leq$ is  the usual order, then
$[[R^{\Bbb Z, \leq}]]$ is isomorphic to $R[[x, x^{-1}]$, the usual Laurent series
ring. Further examples are given in \mcite{Ri1,Ri3}. Results for rings of
generalized power series can be found in \mcite{E-R,L-A,Ri1,Ri2,Ri3,Ri4}.

 Let $S$
be a monoid with subsets $S_1$ and $S_2$ and the disjoint union
$$S=S_1 \dcup S_2.$$ Define a map $P:[[R^{S,\leq}]]\longrightarrow
[[R^{S,\leq}]]$ via
\begin{equation} P(f)(s)=\begin{cases} f(s) & s\in S_1\\ 0
& s\in S_2
\end{cases} \qquad \forall f\in [[R^{S,\leq}]].
\mlabel{eq:subsom}
\end{equation}
Note that $\supp(P(f))\subseteq \supp(f)$. So $\supp(P(f))$ is
artinian and narrow. Thus $P(f)$ is indeed in $[[R^{S,\leq}]]$.

\begin{theorem} $([[R^{S,\leq}]], P)$ is a Rota-Baxter algebra if and only if
$S_1$ and $S_2$ are subsemigroups of $S$.
\mlabel{thm:ssg}
\end{theorem}

\begin{proof}
($\Rightarrow$)
Suppose that $S_1$ and $S_2$ are subsemigroups. For any $f, g\in [[R^{S,\leq}]]$ we will verify
\begin{equation}
P(f)P(g)=P(fP(g))+P(P(f)g)-P(fg).
\mlabel{eq:rba2}
\end{equation}
To check this for $s\in S$, we consider the two cases of $s\in S_1$ and $s\in S_2$.

\noindent
{\bf Case 1. }
Suppose that $s\in S_1$. Then
\begin{eqnarray*} P(fP(g))(s)&=&(fP(g))(s) =\sum_{(u,v)\in X_s(f, P(g))}f(u)P(g)(v) \\&=&\sum_{\substack{(u,v)\in X_s(f, P(g))\\
v\in S_1}}f(u)P(g)(v)=\sum_{\substack{(u,v)\in X_s(f, g)\\
v\in S_1}}f(u)g(v),
\end{eqnarray*}
\begin{eqnarray*} P(P(f)g)(s)&=&(P(f)g)(s)=\sum_{(u,v)\in
X_s(P(f), g)}P(f)(u)g(v)\\&=&\sum_{\substack{(u,v)\in X_s(P(f), g)\\
u\in S_1}}P(f)(u)g(v)=\sum_{\substack{(u,v)\in X_s(f, g)\\
u\in S_1}}f(u)g(v),
\end{eqnarray*}
\begin{eqnarray*} (P(f)P(g))(s)&=& \sum_{(u,v)\in
X_s(P(f), P(g))}P(f)(u)P(g)(v)\\&=&\sum_{\substack{(u,v)\in X_s(P(f), P(g))\\
u\in S_1, v\in S_1}}P(f)(u)P(g)(v)=\sum_{\substack{(u,v)\in X_s(f, g)\\
u\in S_1, v\in S_1}}f(u)g(v).
\end{eqnarray*}
Since $S_2$ is a subsemigroup, we have $\{(u, v)\in X_s(f, g)| u,
v\in S_2\}=\emptyset$. Thus
\begin{eqnarray*} P(fg)(s)&=&(fg)(s)=\sum_{(u,v)\in
X_s(f,
g)}f(u)g(v)\\&=&P(fP(g))(s)+P(P(f)g)(s)-(P(f)P(g))(s).\end{eqnarray*}
\medskip

\noindent
{\bf Case 2. Suppose that $s\in S_2$.} Then $$(P(f)P(g))(s)= \sum_{(u,v)\in
X_s(P(f), P(g))}P(f)(u)P(g)(v)=\sum_{\substack{(u,v)\in X_s(P(f),
P(g))\\ u,v \in S_1}}P(f)(u)P(g)(v)=0$$ since $S_1$ is a subsemigroup. By the definition of $P$, we also have
$$P(fP(g))(s)=P(P(f)g)(s)=P(fg)(s)=0.$$
Thus
$$P(fg)(s)=P(fP(g))(s)+P(P(f)g)(s)-(P(f)P(g))(s).$$

Therefore Eq.~(\mref{eq:rba2}) is verified and
$([[R^{S,\leq}]], P)$ is a Rota-Baxter algebra.

\medskip

($\Leftarrow$)
Conversely, suppose that $S_1$ or $S_2$ is not a subsemigroup of $S$. We show that $P$ is not a Rota-Baxter operator.
\smallskip

\noindent
{\bf Case 1. Suppose $S_1$ is not a subsemigroup of $S$.}
Then there exist $u, v\in S_1$ such that $u+v\not\in S_1$. Thus $u+v\in S_2$. For $w\in S$, define
$e_w\in [[R^{S,\leq}]]$ by
$$ e_w(s) = \begin{cases} 1, s=w, \\ 0, s\neq w. \end{cases}$$
Then we have
\begin{eqnarray*} P(e_u)(s)&=&\begin{cases} e_u(s), & s\in S_1\\
0, & s\in S_2 \end{cases}\\
&=&\begin{cases} 1, & s=u\\
0, & s\neq u \end{cases}\\
&=& e_u(s).\end{eqnarray*} Thus $P(e_u)=e_u$. Similarly
$P(e_v)=e_v$. Hence
$$(P(e_uP(e_v))+P(P(e_u)e_v)-P(e_ue_v))(u+v)=P(e_ue_v)(u+v)=0$$
since $u+v\in S_2$. On the other hand,
$$(P(e_u)P(e_v))(u+v)=(e_ue_v)(u+v)=\sum_{(u',v')\in
X_{u+v}(e_u, e_v)}e_u(u')e_v(v')=1.$$ Thus $P(e_u)P(e_v)\neq
P(e_uP(e_v))+P(P(e_u)e_v)-P(e_ue_v).$ So $([[R^{S,\leq}]], P)$ is not a
Rota-Baxter algebra.
\medskip

\noindent
{\bf Case 2. Suppose $S_2$ is not a subsemigroup of $S$.}
If $([[R^{S,\leq}]], P)$ were a Rota-Baxter algebra, then as is well-known,
$([[R^{S,\leq}]], \tilde{P})$ is also a Rota-Baxter algebra. Here
$$\tilde{P}:  [[R^{S,\leq}]]\longrightarrow [[R^{S,\leq}]]$$
is defined by
$$\tilde{P}(f)(s)=((\id-P)(f))(s)
=\begin{cases} 0, & s\in S_1\\ f(s),
& s\in S_2
\end{cases}\qquad \forall f\in [[R^{S,\leq}]].$$
Applying to $\tilde{P}$ the only if part of the theorem that we have proved above, we conclude that $S_1$ and $S_2$ are subsemigroups of $S$. This is a contradiction.
Therefore, $([[R^{S,\leq}]], P)$ is not a Rota-Baxter algebra.
\end{proof}

Recall that a semigroup $S$ is a strictly ordered monoid with the discrete order and the corresponding generalized power series $[[R^{S,\leq}]]$ is simply the semigroup ring $R\,S$.
\begin{coro}
Let $S$ be a semigroup with a
disjoint union $$ S=S_1 \dcup S_2.$$
Define
$P:R\,S \to R\,S$ by
\begin{equation} P(x)=\left \{ \begin{array}{ll} x, & x\in S_1, \\ 0, & x\in S_2 \end{array} \right .
\mlabel{eq:subsg}
\end{equation}
Then $(R\,S, P)$ is a Rota-Baxter algebra if and only if $S_1$ and $S_2$ are subsemigroups of $S$.
\end{coro}

Combining this corollary with Theorem~\mref{thm:ssg}, we have
\begin{coro}
Let $S$ be a strictly ordered monoid with a
disjoint union $$ S=S_1 \dcup S_2.$$
Define $P:R\,S \to R\,S$ as in Eq.~(\mref{eq:subsg}) and $\widehat{P}:[[R^{S,\leq}]]\to [[R^{S,\leq}]]$ as in Eq.~(\mref{eq:subsom}).
Then $P$ is a Rota-Baxter operator if and only if
$\widehat{P}$ is a Rota-Baxter operator.
\end{coro}

\begin{remark} {\rm
Suppose that  $S=S_1 \dcup S_2=T_1 \dcup T_2$ where
$S_i$ and $T_i$ are subsemigroups of $S$, $i=1, 2$. Then
$([[R^{S,\leq}]], P)$ and $([[R^{S,\leq}]], Q)$ are Rota-Baxter
algebras where $P$ and $Q$ are defined as in Eq.~(\mref{eq:subsg}). It is easy to see
that $P\circ Q=Q\circ P$ as linear operators on $[[R^{S,\leq}]]$. An ennea algebra~\mcite{Le1} is a $R$-module with 9 binary multiplications that satisfy 49 relations. In~\mcite{E-G2}, it is shown that the operad of ennea algebras is the black square product of
the operad of dendriform trialgebras with itself.
By Corollary 2.6 in~\mcite{Le1}, an algebra with two commuting Rota-Baxter operators of weight $-1$ is naturally an ennea algebra.
Hence we get an example of an ennea algebra.
}
\end{remark}

\section{Cut-off operators}
\mlabel{sec:cutoff}

We now study when a cut-off operator on a general power series ring is a Rota-Baxter operator. Let $w\in S$ be given. The {\bf cut-off operator at $w$} on $[[R^{S,\leq}]]$ is the linear operator
$P_w:[[R^{S,\leq}]]\longrightarrow [[R^{S,\leq}]]$ defined by
\begin{eqnarray*} P_w(f): S & \longrightarrow & R,\\
s &\mapsto& \begin{cases} f(s),& s<w\\ 0, & s\not< w \end{cases}
\end{eqnarray*}
for every $f\in [[R^{S,\leq}]]$. Clearly $\supp(P_w(f))\subseteq
\supp(f)$. So $\supp(P_w(f))$ is artinian and narrow. Thus
$P_w(f)\in[[R^{S,\leq}]]$.
Denote
$$A_w=\{(u, v)|u, v\in S, u\not< w, v\not< w, u+v< w\},$$
$$B_w=\{(u, v)|u, v\in S, u<w, v<w, u+v \not<   w\}.$$

\begin{prop} Let  $(S, \leq)$ be a strictly
 ordered monoid. Then $([[R^{S,\leq}]], P_w)$ is a Rota-Baxter
algebra if and only if $ A_w= \emptyset$ and $B_w= \emptyset$.
\mlabel{pp:set}
\end{prop}

\begin{proof}
Note that $P_w$ is the operator $P$ defined in Eq.~(\mref{eq:subsom}) with $S_1=\{s\in S| s<w\}$ and
$S_2=\{s\in S| s\not< w\}$.
Then the proposition follows from Theorem~\mref{thm:ssg} because $A_w=\emptyset$ if and only if $\{s\in S| s\not< w\}$
is a subsemigroup of $S$, and $B_w=\emptyset$ if and only if $\{s\in
S| s< w\}$ is a subsemigroup of $S$.
\end{proof}

\begin{coro} Let  $(S, \leq)$ be a strictly
 totally ordered monoid. Then $([[R^{S,\leq}]], P_w)$ is a
Rota-Baxter algebra if and only if $w\geq 0$ and $B_w= \emptyset$.
\mlabel{co:strict}
\end{coro}
\begin{proof}
We just need to show that $A_w=\emptyset$ if and only if $w\geq 0$.

Note that since $\leq$ is a total order on $S$,
$s\not<t$ if and only if $t\leq s$. If $w<0$, then clearly $(w,
w)\in A_w$ and, so $A_w\neq \emptyset$. Conversely suppose that $(u,
v)\in A_w$. Then $w\leq u, w\leq v, u+v< w$. Thus $w+w\leq u+v<w$
and, so $w<0$.
\end{proof}
\begin{coro} Consider the strictly
totally ordered monoid $(\ZZ,\leq)$. Then $([[R^{\ZZ,\leq}]], P_w)$ is a Rota-Baxter algebra if and only if $w=0,1$.
\end{coro}
\begin{proof}
This follows from Corollary~\mref{co:strict} and the proof of Proposition~\mref{pp:set} since
\begin{align*}
B_w=\emptyset & \Leftrightarrow
\{s\in
S| s< w\}\ \mbox{ is a subsemigroup of } S \\
& \Leftrightarrow \quad w=0,1
\end{align*}
\end{proof}
Thus we have recovered the Rota-Baxter algebra $(R[[\vep,\vep^{-1}],P)$ of Laurent series that has motivated our study.
\medskip

\noindent
{\bf Acknowledgements} This work was supported by
NSF grant DMS 0505445 of U.S (Li Guo) and
by TRAPOYT and the Cultivation Fund of the Key Scientific and
Technical Innovation Project, Ministry of Education of China (Zhongkui Liu).


%
%

\end{document}